\documentclass[twoside,12pt,leqno]{article}
\usepackage{amsmath, amsthm, amsfonts, amssymb, color}
\usepackage{mathrsfs}
\usepackage{fancyhdr}

\newcommand{\be}{\begin{eqnarray}}
\newcommand{\ee}{\end{eqnarray}}
\newcommand{\ce}{\begin{eqnarray*}}
\newcommand{\de}{\end{eqnarray*}}

\newtheorem{thm}{Theorem}[section]

\newtheorem{lem}[thm]{Lemma}
\newtheorem{prp}[thm]{Proposition}

\newtheorem{exa}[thm]{Example}

\theoremstyle{definition}
\newtheorem{defn}{Definition}[section]

\definecolor{wco}{rgb}{0.5,0.2,0.3}

\numberwithin{equation}{section}
\theoremstyle{remark}
\newtheorem{rem}{Remark}[section]

\def\<{\langle} \def\>{\rangle}

 \def\beq{\begin{equation}}

 \def\ee{\varepsilon}

\def\[{{\Big[}}
\def\]{{\Big]}}

\def\({{\Big(}}
\def\){{\Big)}}

\title{{\bf  Existence and Uniqueness of Solutions to Nonlinear Evolution Equations
 with Locally Monotone Operators}
\footnote{Supported in part by DFG--Internationales Graduiertenkolleg
``Stochastics and Real World Models'', the SFB-701 and the BiBoS-Research Center. 
The support of  Issac Newton Institute for Mathematical 
Sciences in Cambridge is also gratefully acknowledged where part of this work was done during the special semester on 
``Stochastic Partial Differential Equations''.} }
\author{{\bf  Wei Liu
\footnote{Corresponding author: wei.liu@uni-bielefeld.de}
}\\
{\footnotesize  Fakult\"at f\"ur Mathematik, Universit\"at Bielefeld,
D-33501 Bielefeld, Germany}\\
}
\date{}
\begin{document}
\maketitle

\begin{abstract}
In this paper we establish the existence and uniqueness of solutions
for nonlinear evolution equations on Banach  space with locally
monotone operators, which is a generalization of the classical
result by J.L. Lions for monotone operators. In particular, we show
that local monotonicity implies the pseudo-monotonicity.
 The main result is applied to various types of PDE such as reaction-diffusion equations, generalized Burgers equation,
  Navier-Stokes equation, 3D Leray-$\alpha$ model and
  $p$-Laplace equation  with non-monotone perturbations.
\end{abstract}

\bigbreak

\section{Main results}

Consider the following Gelfand triple
$$V \subseteq H\equiv H^*\subseteq V^*,$$
 $i.e.$  $(H, \<\cdot,\cdot\>_H)$ is a real separable
Hilbert space and identified with its dual space $H^*$ by the Riesz
isomorphism, $V$ is a real reflexive  Banach space such that it is
 continuously and densely embedded into $H$. If  $\<\cdot,\cdot\>_V$ denotes  the
 dualization
between  $V$ and its dual space $V^*$,  then it follows that
$$ \<u, v\>_V=\<u, v\>_H, \ \  u\in H ,v\in V.$$
The main aim of this paper is to establish the existence and uniqueness of
solutions for general nonlinear evolution equations
\begin{equation}\label{1.1}
 u'(t)=A(t,u(t))+b(t), \ 0<t<T,  \ u(0)=u_0\in H,
\end{equation}
where $T>0$, $u'$ is the generalized derivative of $u$ on $(0,T)$ and $A:[0,T]\times V\rightarrow V^*, b:[0,T]\rightarrow V^*$ is measurable, $i.e.$  for each
$u\in L^1([0,T]; V)$, $A(t,u(t))$ is $V^*$-measurable on $[0,T]$.

It's well known that (\ref{1.1}) has a unique solution if $A$
satisfies the monotone and coercivity conditions (cf.
\cite{Li69,KR79,Z90}). The proof is mainly based on the Galerkin
approximation and the monotonicity tricks. The theory of monotone
operators started from the substantial work of Minty
\cite{Mi62,Mi63}, then it was studied systematically by Browder
\cite{Bro63,Bro64} in order to obtain existence theorems for
quasi-linear elliptic and parabolic partial differential equations.
The existence results of Browder were generalized to more general
classes of quasi-linear elliptic differential equations by Leray and
Lions \cite{LL65}, and Hartman and Stampacchia \cite{HS66}.  We
refer to \cite{Br73,Li69,Sh97,Z90} for more detailed exposition and
references.

One of most important extensions of monotone operator is the pseudo-monotone operator, which was
 first introduced
 by Br\'{e}zis in \cite{Br68}. The prototype of a pseudo-monotone operator is the sum of a monotone operator and
a strongly continuous operator ($i.e.$ a
 operator maps a weakly convergent sequence into a strongly convergent sequence). Hence
the theory of
pseudo-monotone operator unifies both the monotonicity arguments and the
compactness arguments. For example, it can be applied to show the existence of  solutions
for general quasi-linear elliptic equations with lower order terms which satisfy no monotonicity condition (cf. \cite{Bro72,Sh97,Z90}).

This variational approach  has also been adapted for analyzing stochastic partial differential equations (SPDE).
 The existence and uniqueness of solutions to SPDE was first developed by Pardoux \cite{Par75}, Krylov and
Rozovskii \cite{KR79}, we refer to \cite{G,RRW} for further
generalizations. Within this framework many different types of
properties have  been established recently, $e.g.$ see
\cite{L08b,RZ} for the small noise large deviation principle,
\cite{GM09} for discretization approximation schemes to the solution
of SPDE,  \cite{L08,LW08,W07} for the  Harnack inequality and
consequent ergodicity, compactness and contractivity for the
associated transition semigroups, and \cite{L10,BGLR,GLR} for the
invariance of subspaces and existence of random attractors for
corresponding random dynamical systems.

 In this work we  establish the existence, uniqueness and continuous dependence on initial conditions
 of solutions to (\ref{1.1})
by using  the local monotonicity condition  instead of the classical
monotonicity condition. The analogous result for
stochastic PDE has been established in \cite{LR10}. The standard
growth condition on $A$ (cf. \cite{Li69,KR79,Z90}) is also replaced by a much weaker condition
 such that the main result
can be applied to larger class of examples. One of the key observations is
that we show the local monotonicity implies the pseudo-monotonicity
(see Lemma \ref{L2.1}), which may have some independent interests.
The main result is applied
 to establish the existence and uniqueness of solutions for a large class of nonlinear evolution equations
such as reaction diffusion equations, generalized Burgers type equations, generalized
 $p$-Laplace equations, 2-D Navier-Stokes
equation and 3D Leray-$\alpha$ model of turbulence.

%\begin{hyp}
  Suppose for  fixed $\alpha>1, \beta\ge 0$  there exist constants $\delta>0$,
  $C$ and
 a positive function $f\in L^1([0,T];  \mathbb{R})$
 such that the following conditions hold for all $t\in[0,T]$ and $v,v_1,v_2\in V$.
 \begin{enumerate}
 \item [$(H1)$] (Hemicontinuity)
      The map  $ s\mapsto \<A(t,v_1+s v_2),v\>_V$ is  continuous on $\mathbb{R}$.

\item[$(H2)$] (Local monotonicity)
     $$  \<A(t,v_1)-A(t, v_2), v_1-v_2\>_V
     \le \left(C+\rho(v_1)+\eta(v_2) \right)  \|v_1-v_2\|_H^2, $$
where $\rho,\eta: V\rightarrow [0,+\infty)$ are measurable functions
and locally bounded  in $V$.

  \item [$(H3)$] (Coercivity)
    $$ 2 \<A(t,v), v\>_V  \le  -\delta
    \|v\|_V^{\alpha}  +C\|v\|_H^2+ f(t).$$

  \item[$(H4)$] (Growth)
   $$ \|A(t,v)\|_{V^*} \le \bigg(  f(t)^{\frac{\alpha-1}{\alpha}} +
   C\|v\|_V^{\alpha-1} \bigg)  \bigg( 1+ \|v\|_H^{\beta} \bigg).$$

\end{enumerate}
%\end{hyp}

\begin{rem}
(1) If $\beta=0$ and $\rho=\eta\equiv 0$, then $(H1)-(H4)$ are the
classical monotone and coercive conditions in \cite[Theorem
30.A]{Z90} (see also \cite{Li69,KR79,PR07}). It can be applied to
many quasi-linear PDE such as porous medium equation and $p$-Laplace
equation (cf. \cite{Z90,PR07}).

(2) One typical form of $(H2)$ in applications  is
$$\rho(v)=\eta(v)=C\|v\|^\gamma,$$
where $\|\cdot\|$ is some norm on $V$ and $C,\gamma$ are some
constants. The typical examples are 2-D Navier-Stokes equation on
bounded or unbounded domain and Burgers equation, which satisfy
$(H2)$ but do not satisfy the classical monotonicity condition
($i.e.$ $\rho=\eta\equiv0$). We refer to section 3 for more
examples.

(3) If $\rho\equiv 0$ in $(H2)$, then the existence and uniqueness
of solutions to (\ref{1.1}) with general random noise has been
established in \cite{LR10} by using some different techniques.

(4) $(H4)$ is also weaker than the following standard growth condition
assumed in the literature (cf. \cite{KR79,Z90,PR07}:
\begin{equation}\label{growth}
 \|A(t,v)\|_{V^*} \le  f(t)^{\frac{\alpha-1}{\alpha}}+
   C\|v\|_V^{\alpha-1} .
\end{equation}
The advantage of $(H4)$ is, $e.g.$, to include many semilinear type
equations with nonlinear perturbation terms. For example, if we
consider the reaction-diffusion type equation, $i.e.$  $A(u)=\Delta
u+F(u)$, then for verifying the coercivity $(H3)$ we have
$\alpha=2$. Hence (\ref{growth}) implies that $F$  has at most
linear growth. However, we  can allow $F$ to have some polynomial
growth by using $(H4)$ here. We refer to section 3 for more details.
\end{rem}

Now we can state the main result, which gives a unified framework to analyze various classes of nonlinear
evolution equations.

\begin{thm}\label{T1}
Suppose that $V \subseteq H$ is compact and $(H1)$-$(H4)$ hold, then  for any $u_0\in H,\ b\in L^{\frac{\alpha}{\alpha-1}}([0,T];V^*)$
 $(\ref{1.1})$ has a  solution
$$u\in L^\alpha([0,T];V)\cap C([0,T];H), \     u'\in  L^{\frac{\alpha}{\alpha-1}}([0,T];V^*)$$
such that
$$ \<u(t), v\>_H=\<u_0,v\>_H + \int_0^t \<A(s,u(s))+b(s), v\>_V d s , \ t\in[0,T], v\in V.  $$
Moreover,  if there exist constants $C$ and $\gamma$ such that
\begin{equation}\label{c3}
  \rho(v)+\eta(v) \le C(1+\|v\|_V^\alpha) (1+\|v\|_H^\gamma), \
v\in V,
\end{equation}
 then the solution of $(\ref{1.1})$ is unique.
\end{thm}

\begin{rem}\label{R1}
(1) The proof is based on  Galerkin approximation. Moreover, by the Lions-Aubin theorem
(cf. \cite[Chapter III, Proposition 1.3]{Sh97}), the compact embedding of $V\subseteq H$ implies the following embedding
$$  W^1_\alpha(0,T;V,H):=\{ u\in L^\alpha([0,T];V): u'\in L^{\frac{\alpha}{\alpha-1}}([0,T];V^*)  \} \subseteq L^\alpha(0,T;H)     $$
is also compact. Hence there exists a subsequence of the solutions of the Galerkin
approximated equations (see (\ref{2.1}) in Section 2) strongly converges to the solution of (\ref{1.1})
in $L^\alpha(0,T;H)$.

 (2) One can easily see from the proof that the solution of $(\ref{1.1})$ is  unique
  if all solutions of $(\ref{1.1})$
satisfy
$$ \int_0^T \left( \rho(u(s)) +\eta(u(s)) \right) d s<\infty.  $$

(3) The compact embedding $V \subseteq H$ is  required in the main
result. For (global) monotonicity one can easily drop this
assumption. In fact, the classical monotonicity tricks only works in
general for the operator satisfies $(H2)$ with $C=\rho=\eta=0$. For
$C>0$ (but $\rho=\eta=0$) one can apply a standard exponential
transformation to (\ref{1.1}) to reduce the case $C>0$ to the case
$C=0$. However, this kind of techniques does not work for the
locally
 monotone case. In order to verify the pseudo-monotonicity of $A(t,\cdot)$, we have to
split it into the sum of $A(t,\cdot)-cI$ and $cI$.  And  $I$ is
pseudo-monotone if and only if the embedding $V \subset H$ is
compact.
% \end{rem}

%%%%%%%%%%%%%%%%%%%%%%%%%%%%%%%%%%%%%%%%%%%%%%%%%%%%%%%%%%%%%%%%%%%%%%%%

(4) We can also establish a similar result for stochastic evolution equations in Hilbert space with additive
noise:
\begin{equation}\label{SDE}
d X(t)=A(t, X(t))dt + B d N(t), \ t\ge 0, \ X(0)=x.
\end{equation}
Here  $A:[0,T]\times V\rightarrow V^*$ and $B\in L(U; H)$, where $U$ is another Hilbert space and
$N(t)$ is  a $U$-valued adapted stochastic process definded on a
filtered probability space $(\Omega,\mathcal{F},\mathcal{F}_t,\mathbb{P})$ (cf. \cite{GLR,PR07}).
By a standard transformation (substitution), (\ref{SDE}) can be reduced to deterministic evolution
 equations with a random parameter which Thoerem \ref{T1} can be applied to. This result and some further applications
will be investigated in a separated paper.
\end{rem}

%%%%%%%%%%%%%%%%%%%%%%%%%%%%%%%%%%%%%%%%%%%%%%%%%%%%%%%%%%%%%%%%%%%%%%%%%%%%%%%%%%%%%%%%%%%5

Next result is the continuous dependence of solution of $(\ref{1.1})$ on $u_0$ and $b$.
\begin{thm}\label{T2}
Suppose that $V \subseteq H$ is compact and $(H1)$-$(H4)$ hold, $u_i$ are the solution of $(\ref{1.1})$
with $u_{i,0}\in H$ and $b_i\in L^{\frac{\alpha}{\alpha-1}}([0,T];V^*)\cap L^{2}([0,T]; H)$, $i=1,2$ respectively and
 satisfy
$$ \int_0^T\left( \rho(u_1(s))+\eta(u_2(s)) \right)  d s<\infty.  $$
Then there exists a constant $C$ such that
\begin{equation}
 \begin{split}
      \|u_1(t)-u_2(t)\|_H^2
\le &   \exp\left[\int_0^t \left(C+\rho(u_1(s))+\eta(u_2(s)) \right) d s  \right]\\
 & \ \  \cdot \left( \|u_{1,0}-u_{2,0}\|_H^2 +\int_0^t \|b_1(s)-b_2(s)\|_H^2 ds \right), \ t\in[0, T].
 \end{split}
\end{equation}
\end{thm}

The paper is organized as follows. The proofs of the main results are given in the next section.
In Section 3 we apply the main results to several concrete semilinear and quasi-linear evolution
equations on Banach space.

\section{Proofs of Main Theorems}
\subsection{Proof of Theorem \ref{T1}}
In order to make the proof easier to follow, we first give the outline of the proof for the reader's convenience.

Step 1:  Galerkin approximation; local monotonicity and coercivity
implies the existence (and uniqueness) of  solutions to the approximated equations;

Step 2: A priori estimates was obtained from coercivity;

Step 3: Verify the weak limits by using modified  monotonicity
tricks;

Step 4: Uniqueness follows from local monotonicity.

The main difficulty is in the third step. The classical monotonicity tricks does not
 work for locally monotone operators.
The crucial part for overcoming this difficulty is the following result:
every locally monotone operator is pseudo-monotone.
Then by using some techniques related with pseudo-monotonicity one can
 establish the existence of  solutions.

We first recall the definition of  pseudo-monotone operator
introduced  first
 by Br\'{e}zis in \cite{Br68}. We use the standard notation ``$\rightharpoonup$''
for weak convergence in Banach space.

\begin{defn} The operator $A: V\rightarrow V^*$ is called pseudo-monotone  if $v_n\rightharpoonup v$ in $V$ and
$$     \liminf_{n\rightarrow\infty} \<A(v_n), v_n-v\>_V\ge 0 $$
implies for all $u\in V$
$$   \<A(v), v-u\>_V \ge  \limsup_{n\rightarrow\infty} \<A(v_n), v_n-u\>_V.  $$
 \end{defn}

\begin{rem} (1)
We remark that the definition of pseudo-monotone operator here coincides with the definition in
\cite{Z90} (one should replace $A$ here by $-A$ in \cite{Z90} due to
different form of the formulation for  evolution equations).

(2) The class of pseudo-monotone operators  is stable under
summation ($i.e.$  the sum of two pseudo-monotone operators is still
pseudo-monotone) and strictly smaller than the class of operator of
type (M) (cf. \cite{Z90,Sh97}). And note that the class of operator
of type (M) is not stable under summation. A counterexample can be
found in \cite{Sh97}.
\end{rem}

\begin{prp}\label{condition P}
 If $A$ is pseudo-monotone, then $v_n\rightharpoonup v$ in $V$ implies that
$$  \liminf_{n\rightarrow\infty}   \<A(v_n), v_n-v\>_V\le 0.         $$
 \end{prp}
\begin{proof}
If the conclusion is not true, then there exists $v_n\rightharpoonup
v$ in $V$ such that
$$  \liminf_{n\rightarrow\infty}   \<A(v_n), v_n-v\>_V> 0.         $$
Then we can extract a subsequence  such that $v_{n_k}\rightharpoonup
v$ and
\begin{equation}\label{contradiction}
  \lim_{k\rightarrow\infty}   \<A(v_{n_k}), v_{n_k}-v\>_V> 0.
\end{equation}
Then by the pseudo-monotonicity of $A$ we have for all $u\in V$
$$  \<A(v), v-u\>_V \ge  \limsup_{k\rightarrow\infty} \<A(v_{n_k}), v_{n_k}-u\>_V.  $$
By taking $u=v$ we obtain
 $$ \limsup_{k\rightarrow\infty} \<A(v_{n_k}), v_{n_k}-v\>_V\le 0,$$
which is a contradiction to (\ref{contradiction}).

Hence the proof is completed.
\end{proof}

\begin{rem}
We also recall a slightly modified definition of  pseudo-monotone operator by Browder
 (cf. \cite{Bro77}):
% \begin{definition*}
The operator $A: V\rightarrow V^*$ is called pseudo-monotone  if
$v_n\rightharpoonup v$ in $V$ and
$$     \liminf_{n\rightarrow\infty} \<A(v_n), v_n-v\>_V\ge 0 $$
implies
$$   A(v_n)\rightharpoonup A(v)\ \    \text{and}   \  \
\lim_{n\rightarrow\infty} \<A(v_n), v_n\>_V=\<A(v), v\>_V.  $$
% \end{definition*}
 From this definition one  clearly see the role of pseudo-monotone operator for
verifying the limit of weakly convergent sequence under
nonlinear operator.

If $A$ is bounded ($i.e.$  $A$ maps bounded set
into  bounded set), then it's easy to show that these two
definitions are equivalent by Proposition \ref{condition P}. In particular, under the
assumption of $(H4)$,  these two
definitions are equivalent.
\end{rem}

\begin{lem}\label{L2.1}
 If the  embedding $V\subseteq H$ is compact, then $(H1)$ and $(H2)$  implies that
$A(t,\cdot)$ is pseudo-monotone for any $t\in[0,T]$.
\end{lem}
\begin{proof}
For simplicity, we denote $A(t,\cdot)$ by $A(\cdot)$ for any fixed $t\in[0,T]$.

 Suppose $v_n\rightharpoonup v$ in $V$ and
\begin{equation}\label{2.0}
 \liminf_{n\rightarrow\infty} \<A(v_n), v_n-v\>_V\ge 0,
\end{equation}
then for any  $u\in V$ we need to show
\begin{equation}\label{e1}
  \<A(v), v-u\>_V \ge  \limsup_{n\rightarrow\infty} \<A(v_n), v_n-u\>_V.
\end{equation}

Given $u$ and the constant $C$ in $(H2)$, we take
$$ K=\|u\|_V+ \|v\|_V+ \sup_n \|v_n\|_V; \ \  C_1=\sup_{v:\|v\|_V\le 2K} \left(C+\rho(v)+\eta(v) \right)< \infty.   $$
 Since the  embedding $V\subseteq H$ is compact, we have
$v_n\rightarrow v$ in $V^*$ and
$$   \<C_{1}v, v-u\>_V =  \lim_{n\rightarrow\infty} \<C_{1}v_n, v_n-u\>_V. $$
Hence for proving (\ref{e1}) it's sufficient to show that
$$   \<A_0(v), v-u\>_V \ge  \limsup_{n\rightarrow\infty} \<A_0(v_n), v_n-u\>_V, $$
where $A_0=A-C_{1}I$ ($I$ is the identity operator).

Then $(H2)$  implies that
$$  \limsup_{n\rightarrow\infty} \<A_0(v_n), v_n-v\>_V\le   \limsup_{n\rightarrow\infty}
\<A_0(v), v_n-v\>_V =0.  $$
By (\ref{2.0}) we obtain
\begin{equation}\label{2.2}
  \lim_{n\rightarrow\infty} \<A_0(v_n), v_n-v\>_V= 0.
\end{equation}
Let $z=v+t(u-v)$ with $t\in(0,\frac{1}{2})$, then
the local monotonicity $(H2)$ implies that
$$ \<A_0(v_n)-A_0(z),  v_n-z\>_V\le 0, $$
$i.e.$
$$ t \<A_0(z), v-u\>_V-(1-t)\<A_0(v_n), v_n-v\>_V \ge t\<A_0(v_n), v_n-u\>_V-\<A_0(z), v_n-v\>_V. $$
By taking $\limsup$ on both sides and using (\ref{2.2}) we have
$$    \<A_0(z), v-u\>_V\ge \limsup_{n\rightarrow\infty} \<A_0(v_n), v_n-u\>_V.  $$
Then letting $t\rightarrow 0$, by the hemicontinuity $(H1)$ we obtain
$$   \<A_0(v), v-u\>_V \ge  \limsup_{n\rightarrow\infty} \<A_0(v_n), v_n-u\>_V.  $$
Therefore, $A$ is pseudo-monotone.
\end{proof}

\begin{rem}
For some concrete operators, the local monotonicity $(H2)$  might be
easier to check by explicit calculations than the definition of
pseudo-monotonicity. Hence the above result can be also seen as a
computable sufficient condition for the pseudo-monotonicity in
applications.
\end{rem}

The proof of  Theorem  \ref{T1} is split into a few lemmas. Let
$X:=L^\alpha([0,T];V)$, then
$X^*=L^{\frac{\alpha}{\alpha-1}}([0,T];V^*)$. We denote by
$W^1_\alpha(0,T;V, H)$ the Banach space
$$ W^1_\alpha(0,T;V, H)=\{ u\in X: u'\in X^*  \}, $$
where the norm is defined by
$$ \|u\|_W:=\|u\|_X+\|u'\|_{X^*}= \left(\int_0^T\|u(t)\|_V^\alpha dt\right)^{\frac{1}{\alpha}}+
\left(\int_0^T\|u'(t)\|_{V^*}^{\frac{\alpha}{\alpha-1}} dt\right)^{\frac{\alpha-1}{\alpha}}. $$
 It's well known that $W^1_\alpha(0,T;V, H)$ is a reflexive
Banach space and it is continuously imbedded into $C([0,T];H)$ (cf. \cite{Z90}). Moreover, we also have the following
integration by parts formula
\begin{equation*}
 \begin{split}
\<u(t), v(t)\>_H -\<u(0), v(0)\>_H=\int_0^t& \<u'(s), v(s)\>_V ds
+ \int_0^t \<v'(s), u(s)\>_V ds,\\
& \ t\in[0,T],  \ u,v\in W^1_\alpha(0,T; V,H).
 \end{split}
\end{equation*}

We first consider the Galerkin approximation to (\ref{1.1}).

 Let $\{e_1,e_2,\cdots \}\subset V$ be an orthonormal basis in $H$ and
 let $H_n:=span\{e_1,\cdots,e_n\}$ such that $span\{e_1,e_2,\cdots\}$ is dense in $V$. Let $P_n:V^*\rightarrow H_n$ be defined by
$$ P_ny:=\sum_{i=1}^n \<y,e_i\>_V e_i, \ y\in V^*.  $$
Obviously, $P_n|_H$ is just the orthogonal projection onto $H_n$ in H and we have
$$ \<P_nA(t,u), v\>_V=\<P_nA(t,u),v\>_H=\<A(t,u),v\>_V, \ u\in V, v\in H_n.  $$

For each finite $n\in \mathbb{N}$ we consider the following evolution equation on $H_n$:
\begin{equation}\label{2.1}
 u_n'(t)=P_nA(t,u_n(t))+P_n b(t), \  0<t<T, \ u_n(0)=P_nu_0\in H_n.
\end{equation}
It is easy to show that $P_nA$ is locally monotone and coercive on $H_n$ (finite dimensional space). According to the classical result of
 Krylov (cf. \cite{K99} or \cite[Theorem 3.1.1]{PR07}), there exists a unique solution $u_n$ to (\ref{2.1}) such that
 $$u_n\in L^\alpha([0,T];H_n)\cap C([0,T];H_n), \      u_n'\in  L^{\frac{\alpha}{\alpha-1}}([0,T];H_n) .$$

\begin{lem}\label{l2.3}

 Under the assumptions of Theorem \ref{T1}, there exists a constant $K>0$ such that
\begin{equation}
 \|u_n\|_X+\sup_{t\in[0,T]}\|u_n\|_H+\|A(\cdot,u_n)\|_{X^*}\le K, \ n\ge 1.
\end{equation}
\end{lem}

\begin{proof} By the integration by parts formula and $(H3)$ we have
\begin{equation}\begin{split}
& ~~~~\|u_n(t)\|_H^2-\|u_n(0)\|_H^2\\
&=2\int_0^t\<u_n'(s),u_n(s)\>_V d s\\
&=2\int_0^t\<P_nA(s,u_n(s))+P_nb(s),u_n(s)\>_V d s\\
&=2\int_0^t\<A(s,u_n(s))+b(s),u_n(s)\>_V d s\\
&\le\int_0^t\left(-\delta\|u_n(s)\|_V^\alpha+C\|u_n(s)\|_H^2+f(s)+\|b(s)\|_{V^*}\|u_n(s)\|_V   \right) d s  \\
&\le \int_0^t\left(-\frac{\delta}{2}\|u_n(s)\|_V^\alpha+C\|u_n(s)\|_H^2+f(s)+C_1\|b(s)\|_{V^*}^{\frac{\alpha}{\alpha-1}}   \right) d s,
\end{split}
\end{equation}
where $C_1$ is a constant induced from Young's inequality.

Hence we have for $t\in [0,T]$,
$$  \|u_n(t)\|_H^2+\frac{\delta}{2}\int_0^t\|u_n(s)\|_V^\alpha d s\le \|u(0)\|_H^2 +
C\int_0^t\|u_n(s)\|_H^2 ds +\int_0^t  \left(f(s)+C_1\|b(s)\|_{V^*}^{\frac{\alpha}{\alpha-1}}  \right) d s.  $$
Then by  Gronwall's lemma we have
$$    \|u_n(t)\|_H^2 \le e^{Ct}\left( \|u(0)\|_H^2
+\int_0^t e^{-Cs} \left(f(s)+C_1\|b(s)\|_{V^*}^{\frac{\alpha}{\alpha-1}}  \right) d s \right), \ t\in[0, T].   $$
$$ \frac{\delta}{2}\int_0^t\|u_n(s)\|_V^\alpha d s\le e^{Ct}\left( \|u(0)\|_H^2 +
\int_0^t e^{-Cs}\left(f(s)+C_1\|b(s)\|_{V^*}^{\frac{\alpha}{\alpha-1}}  \right) d s \right), \ t\in[0, T].  $$
Therefore, there exists a constant $C_2$ such that
$$ \|u_n\|_{X}+\sup_{t\in[0,T]}\|u_n(t)\|_H \le C_2, \ n\ge 1.  $$
Then by $(H4)$ there exists a constant $C_3$ such that
$$ \|A(\cdot, u_n)\|_{X^*} \le C_3, \ n\ge 1.   $$
Hence the proof is complete.
\end{proof}

Note that $X,X^*$ and $H$ are reflexive spaces, by the estimates in Lemma \ref{l2.3}, there exists a subsequence, again denote
by $u_n$, such that as $n\rightarrow\infty$
\begin{equation*}
 \begin{split}
    u_n &\rightharpoonup u\  \ \text{in}\  X \ (\text{also in} \ \   W^1_\alpha(0,T;V,H)); \\
 A(\cdot,u_n)&\rightharpoonup w \ \ \text{in} \ X^*; \\
 u_n(T) &\rightharpoonup z \ \ \text{in}  \  H.
 \end{split}
\end{equation*}

Recall that $u_n(0)=P_nu_0\rightarrow u_0$ in $H$ as $n\rightarrow\infty$.

\begin{lem}
Under the assumptions of Theorem \ref{T1},
 the limit elements $u,w$ and $z$ satisfy  $u\in W^1_\alpha(0,T;V,H)$ and
$$ u'(t)=w(t)+b(t), \ 0<t<T, \ u(0)=u_0, \ u(T)=z.$$
\end{lem}
\begin{proof}
 The proof is standard (cf. \cite[Lemma 30.5]{Z90}), we include it here for the completeness.
Recall the following integration by parts formula
$$\<u(T), v(T)\>_H -\<u(0), v(0)\>_H=\int_0^T \<u'(t), v(t)\>_V dt
+ \int_0^T \<v'(t), u(t)\>_V dt, \ u,v\in W^1_\alpha(0,T; V,H).  $$
Then, for $\psi\in C^\infty([0,T])$ and $v\in H_n$, by (\ref{2.1}) we have
\begin{equation}
 \begin{split}
 & \<u_n(T), \psi(T)v\>_H -\<u_n(0), \psi(0)v\>_H \\
=&\int_0^T \<u_n'(t), \psi(t)v\>_V dt
+ \int_0^T \<\psi'(t)v, u_n(t)\>_V dt\\
=&\int_0^T \<A(t,u_n(t))+b(t), \psi(t)v\>_V dt
+ \int_0^T \<\psi'(t)v, u_n(t)\>_V dt.
 \end{split}
\end{equation}
Letting $n\rightarrow\infty$ we obtain for all $v\in \bigcup_{n} H_n$,
\begin{equation}\label{part}
 \<z, \psi(T)v\>_H -\<u_0, \psi(0)v\>_H
=\int_0^T \<w(t)+b(t), \psi(t)v\>_V dt
+ \int_0^T \<\psi'(t)v, u(t)\>_V dt.
\end{equation}
Since $\bigcup_{n} H_n$ is dense in $V$, it's easy to show that (\ref{part})
hold for all $v\in V, \psi\in C^\infty([0,T])$.

If $\psi(T)=\psi(0)=0$, then we have
$$ \int_0^T \<w(t)+b(t), v\>_V \psi(t) dt
=- \int_0^T \<u(t), v\>_V\psi'(t) dt.   $$
This implies that $u'=w+b, t\in(0,T)$. In particular, we have $u\in W^1_\alpha(0,T;V,H)$.

Then by the integration by parts formula we have
\begin{equation*}\begin{split}
& \<u(T), \psi(T)v\>_H -\<u(0), \psi(0)v\>_H \\
=&\int_0^T \<u'(t), \psi(t)v\>_V dt
+ \int_0^T \<\psi'(t)v, u(t)\>_V dt  \\
  =&\int_0^T \<w(t)+b(t), \psi(t)v\>_V dt
+ \int_0^T \<\psi'(t)v, u(t)\>_V dt.
 \end{split}
\end{equation*}
Hence by (\ref{part}) we obtain
$$   \<u(T), \psi(T)v\>_H -\<u(0), \psi(0)v\>_H =   \<z, \psi(T)v\>_H -\<u_0, \psi(0)v\>_H. $$
Then by choosing  $\psi(T)=1, \psi(0)=0$ and $\psi(T)=0, \psi(0)=1$ respectively  we obtain that
$$   u(T)=z, \  u(0)=u_0.   $$
Hence the proof is complete.
\end{proof}

Next lemma is very crucial for the proof of Theorem \ref{T1}. The result
basically says that $A$ is also a pseudo-monotone operator from $X$ to
$X^*$, hence one can still use a modified monotonicity tricks to verify the limit of the
Galerkin approximation as a solution to (\ref{1.1}).
 The techniques used in the proof is inspired by the works of Hirano and Shioji \cite{H89,S97}.

\begin{lem}\label{L2.5}
Under the assumptions of Theorem \ref{T1},
suppose that
\begin{equation}\label{2.3}
 \liminf_{n\rightarrow \infty} \int_0^T\<A(t,u_n(t)), u_n(t)\>_V d t \ge
\int_0^T \<w(t), u(t)\>_V d t,
\end{equation}
then for any $v\in X$ we have
\begin{equation}
\int_0^T\<A(t,u(t)), u(t)-v(t)\>_V d t \ge
 \limsup_{n\rightarrow \infty}\int_0^T \<A(t,u_n(t)), u_n(t)-v(t)\>_V d t.
\end{equation}

\end{lem}

\begin{proof}

Since $W^1_\alpha(0,T;V,H)\subset C([0,T];H)$ is a continuous embedding, we have that
 $u_n(t)$ converges to $u(t)$ weakly in $H$
for all $t\in[0,T]$. Hence  $u_n(t)$ also converges to $u(t)$ weakly in $V$
for all $t\in[0,T]$.

\textbf{Claim 1:}  For all  $t\in[0,T]$ we have
\begin{equation}\label{claim 1}
  \limsup_{n\rightarrow\infty}\<A(t,u_n(t)),u_n(t)-u(t)\>_V \le 0 .
\end{equation}

Suppose there exists a $t_0$ such that
$$   \limsup_{n\rightarrow\infty}\<A(t_0,u_n(t_0)),u_n(t_0)-u(t_0)\>_V > 0.   $$
Then we can take a subsequence such that
$$   \lim_{i\rightarrow\infty}\<A(t_0,u_{n_i}(t_0)),u_{n_i}(t_0)-u(t_0)\>_V > 0.   $$
Note that $u_{n_i}(t_0)$ converges to $u(t_0)$ weakly in $V$ and $A(t_0,\cdot)$ is pseudo-monotone, we have
$$ \<A(t_0,u(t_0)), u(t_0)-v\>_V \ge \limsup_{i\rightarrow\infty}\<A(t_0,u_{n_i}(t_0)),u_{n_i}(t_0)-v\>_V,
\ v\in V. $$
In particular, we have
 $$ \limsup_{i\rightarrow\infty}\<A(t_0,u_{n_i}(t_0)),u_{n_i}(t_0)-u(t_0)\>_V\le 0, $$
which is a contradiction with the definition of this subsequence.

Hence (\ref{claim 1}) holds.

By $(H3)$ and $(H4)$ there exists a constant $K$ such that
\begin{equation*}
 \begin{split}
   \<A(t,u_n(t)),u_n(t)-v(t)\>_V\le & -\frac{\delta}{2}\|u_n(t)\|_V^\alpha +K\left( f(t)
+\|u_n(t)\|_H^2  \right)\\
& +K\left(1+\|u_n(t)\|_H^{\alpha \beta} \right)\|v(t)\|_V^\alpha,\
v\in X.
 \end{split}
\end{equation*}

Then by Lemma \ref{l2.3}, Fatou's lemma, (\ref{2.3}) and (\ref{claim 1}) we have
\begin{equation}
\begin{split}
0&\le \liminf_{n\rightarrow \infty}\int_0^T \<A(t,u_n(t)), u_n(t)-u(t)\>_V d t \\
&\le \limsup_{n\rightarrow \infty}\int_0^T \<A(t,u_n(t)), u_n(t)-u(t)\>_V d t \\
&\le \int_0^T \limsup_{n\rightarrow \infty} \<A(t,u_n(t)), u_n(t)-u(t)\>_V d t \le 0.
\end{split}
\end{equation}

Hence
$$   \lim_{n\rightarrow \infty}\int_0^T \<A(t,u_n(t)), u_n(t)-u(t)\>_V d t = 0.   $$

\textbf{Claim 2:} There exists a subsequence $\{u_{n_i}\}$ such that
\begin{equation}\label{claim 2}
 \lim_{i\rightarrow \infty} \<A(t,u_{n_i}(t)), u_{n_i}(t)-u(t)\>_V  = 0 \  \ \text{for}\ a.e. \ t\in[0,T].
\end{equation}
Define $ g_n(t)= \<A(t,u_{n}(t)), u_{n}(t)-u(t)\>_V, \ t\in[0,T] $, then
$$     \lim_{n\rightarrow \infty}\int_0^T g_n(t) d t = 0, \ \ \limsup_{n\rightarrow\infty} g_n(t)\le 0,\ \  t\in[0,T].       $$
Then by Lebesgue's dominated convergence theorem we have
$$ \lim_{n\rightarrow \infty}\int_0^T g_n^+(t) d t = 0,  $$
where $g_n^+(t):=max\{g_n(t), 0 \}$.

Note that $|g_n(t)|=2g_n^+(t)-g_n(t)$, hence we have
$$  \lim_{n\rightarrow \infty}\int_0^T |g_n(t)| d t = 0.     $$
Therefore, we can take a subsequence $\{g_{n_i}(t)\}$ such that
$$     \lim_{i\rightarrow \infty} g_{n_i}(t) = 0 \ \ \text{for}\ a.e. \ t\in[0,T],  $$
$i.e.$  (\ref{claim 2}) holds.

Therefore, for any $v\in X$, we can choose a subsequence $\{u_{n_i}\}$ such that
$$   \lim_{i\rightarrow \infty}\int_0^T \<A(t,u_{n_i}(t)), u_{n_i}(t)-v(t)\>_V d t =
\limsup_{n\rightarrow \infty}\int_0^T \<A(t,u_n(t)), u_n(t)-v(t)\>_V d t;         $$
$$    \lim_{i\rightarrow \infty} \<A(t,u_{n_i}(t)), u_{n_i}(t)-u(t)\>_V  = 0 \ \ \text{for}\ a.e. \ t\in[0,T].       $$
Since $A$ is pseudo-monotone, we have
$$  \<A(t,u(t)), u(t)-v(t)\>_V \ge  \limsup_{i\rightarrow \infty} \<A(t,u_{n_i}(t)), u_{n_i}(t)-v(t)\>_V,\  \ t\in[0,T].       $$
By Fatou's lemma we obtain
\begin{equation}
 \begin{split}
  \int_0^T\<A(t,u(t)), u(t)-v(t)\>_V d t &\ge
 \int_0^T\limsup_{i\rightarrow \infty} \<A(t,u_{n_i}(t)), u_{n_i}(t)-v(t)\>_V d t\\
&\ge  \limsup_{i\rightarrow \infty}\int_0^T \<A(t,u_{n_i}(t)), u_{n_i}(t)-v(t)\>_V d t\\
&= \limsup_{n\rightarrow \infty}\int_0^T \<A(t,u_n(t)), u_n(t)-v(t)\>_V d t.
\end{split}
\end{equation}
Hence the proof is complete.
\end{proof}

\noindent\textbf{Proof of Theorem \ref{T1}}
(i) Existence:
The integration by parts formula implies that
$$ \|u_n(T)\|_H^2-\|u_n(0)\|_H^2=2 \int_0^T \<A(t,u_n(t))+b(t), u_n(t)\>_V d t;   $$
$$ \|u(T)\|_H^2-\|u(0)\|_H^2=2 \int_0^T \<w(t)+b(t), u(t)\>_V d t.   $$
Since $u_n(T)\rightharpoonup z$ in $H$, by the lower semicontinuity of $\|\cdot\|_H$ we have
$$  \liminf_{n\rightarrow \infty} \|u_n(T)\|_H^2\ge \|z\|_H^2=\|u(T)\|_H^2.   $$
Hence we have
\begin{equation*}
 \begin{split}
 \liminf_{n\rightarrow \infty} \int_0^T\<A(t,u_n(t)), u_n(t)\>_V d t &\ge \frac{1}{2}
\left( \|u(T)\|_H^2 -\|u(0)\|_H^2 \right)-\int_0^T\<b(t),u(t)\>_V dt\\
&=\int_0^T \<w(t), u(t)\>_V d t.
\end{split}
\end{equation*}

By Lemma \ref{L2.5}  we have for any $v\in X$,
\begin{equation*}
 \begin{split}
  \int_0^T\<A(t,u(t)), u(t)-v(t)\>_V d t
&\ge\limsup_{n\rightarrow \infty}\int_0^T \<A(t,u_n(t)), u_n(t)-v(t)\>_V d t\\
&\ge\liminf_{n\rightarrow \infty}\int_0^T \<A(t,u_n(t)), u_n(t)-v(t)\>_V d t\\
&\ge \int_0^T \<w(t), u(t)\>_V d t- \int_0^T \<w(t), v(t)\>_V d t\\
&=\int_0^T \<w(t), u(t)-v(t)\>_V d t.
\end{split}
\end{equation*}
Since $v\in X$ is arbitrary, we have $A(\cdot,u)=w$ as the element in $X^*$.

Hence $u$ is a solution to (\ref{1.1}).

(ii) Uniqueness:  Suppose $u(\cdot,u_0),v(\cdot,v_0)$ are the
solutions to (\ref{1.1}) with starting points $u_0,v_0$
respectively, then by the integration by parts formula   we have for
$t\in[0, T]$,
\begin{equation*}
 \begin{split}
  \|u(t)-v(t)\|_H^2&=\|u_0-v_0\|_H^2+ 2\int_0^t\< A(s,u(s))-A(s,v(s)),u(s)-v(s)\>_V ds\\
 &\le \|u_0-v_0\|_H^2+ 2\int_0^t \left(C+\rho(u(s))+ \eta(v(s))    \right)   \|u(s)-v(s)\|_H^2 ds.
 \end{split}
\end{equation*}
By (\ref{c3}) we know that
$$       \int_0^T \left(C+\rho(u(s))+ \eta(v(s))  \right)  d s <\infty.                $$
Then by Gronwall's lemma we obtain
\begin{equation}\label{estimate of difference}
    \|u(t)-v(t)\|_H^2\le\|u_0-v_0\|_H^2 \exp\left[2\int_0^t \left(C+\rho(u(s))+ \eta(v(s))    \right) d s  \right], \ t\in[0, T].
\end{equation}
In particular, if $u_0=v_0$, this  implies the  uniqueness of the
solution of $(\ref{1.1})$. \qed

%%%%%%%%%%%%%%%%%%%%%%%%%%%%%%%%%%%%%%%%%%%%%%%%%%%%%%%%%%%%%%%%%%

%%%%%%%%%%%%%%%%%%%%%%%%%%%%%%%%%%%%%%%%%%%%%%%%%%%%%%%%%%%%%%%%%%%%%%%%%%%%

\subsection{Proof of Theorem \ref{T2}} By $(H2)$ we have
\begin{equation*}
 \begin{split}
  \|u_1(t)-u_2(t)\|_H^2&=\|u_{1,0}-u_{2,0}\|_H^2+ 2\int_0^t\< A(s,u_1(s))-A(s,u_2(s)),u_1(s)-u_2(s)\>_V ds\\
                          & \ \ + 2\int_0^t\< b_1(s)-b_2(s),u_1(s)-u_2(s)\>_V ds\\
 &\le \|u_{1,0}-u_{2,0}\|_H^2+ \int_0^t \| b_1(s)-b_2(s)\|_H^2 d s\\
  &\ \   +\int_0^t \left(C+\rho(u_1(s))+ \eta(u_2(s))   \right) \|u_1(s)-u_2(s)\|_H^2 ds, \ t\in[0, T],
 \end{split}
\end{equation*}
where $C$ is a constant.

Then by Gronwall's lemma we have
\begin{equation*}
 \begin{split}
      \|u_1(t)-u_2(t)\|_H^2
\le &  \exp\left[\int_0^t \left(C+\rho(u_1(s))+ \eta(u_2(s))\right) d s  \right]\\
& \cdot  \left( \|u_{1,0}-u_{2,0}\|_H^2 +\int_0^t
\|b_1(s)-b_2(s)\|_H^2 ds \right), \ t\in[0, T].
 \end{split}
\end{equation*}
\qed

\section{Application to examples}
It's obvious that  the main results can be applied to  nonlinear
evolution equations with monotone operators ($e.g.$
  porous medium equation, $p$-Laplace equation) perturbated by some
non-monotone terms ($e.g.$ some locally Lipschitz perturbation).
Moreover, we  also formulate some examples where the coefficients
are only locally monotone. For simplicity here we only formulate the
examples where the coefficients are time independent, but one can
easily adapt all these examples to the time dependent case.

 Here we use the notation $D_i$ to denote the spatial derivative $\frac{\partial}{\partial x_i}$,
$\Lambda \subseteq \mathbb{R}^d$ is an open bounded domain with smooth
boundary.
 For standard Sobolev space $W_0^{1,p}(\Lambda)$ $(p\ge 2)$  we always use the following (equivalent) Sobolev norm:
$$    \|u\|_{1,p}:=\left(\int_\Lambda |\nabla u(x)|^p d x\right)^{\frac{1}{p}}.      $$

As preparation we first give a lemma for verifying $(H2)$.

\begin{lem}\label{L3.1}
Consider the Gelfand triple
$$ V:=W_0^{1,2}(\Lambda)\subseteq H:=L^2(\Lambda) \subseteq  W^{-1,2}(\Lambda)  $$
and the operator
$$ A(u)=\Delta u+ \sum_{i=1}^d f_i(u)D_i u,  $$
where $f_i$ ($i=1,\cdots,d$) are  Lipschitz functions on $\mathbb{R}$.

(1) If $d<3$,  then there exists a constant
$K$ such that
$$2 \<A(u)-A(v), u-v\>_V
   \le - \|u-v\|_V^2+  \left(K +K\|u\|_{L^4}^4 + K\|v\|_V^2  \right)\|u-v\|_H^2,\ u,v\in V.$$
In particular, if $f_i$ are  bounded functions for $i=1,\cdots,d$, then we have
$$2 \<A(u)-A(v), u-v\>_V
   \le - \|u-v\|_V^2+  \left(K  + K\|v\|_V^2  \right)\|u-v\|_H^2,\ u,v\in V.$$

(2) For $d=3$ we have
$$2 \<A(u)-A(v), u-v\>_V
   \le - \|u-v\|_V^2+  \left(K +K\|u\|_{L^4}^8+ K\|v\|_V^4 \right)\|u-v\|_H^2,\ u,v\in V.$$
In particular, if $f_i$ are   bounded functions for $i=1,\cdots,d$,  we have
$$2 \<A(u)-A(v), u-v\>_V
   \le - \|u-v\|_V^2+  \left(K + K\|v\|_V^4 \right)\|u-v\|_H^2,\ u,v\in V.$$

(3) If $f_i$ are bounded measurable functions on $\Lambda$ and independent of $u$ for $i=1,\cdots,d$,
i.e.
$$ A(u)=\Delta u+ \sum_{i=1}^d f_i\cdot D_i u,  $$
then for any $d\ge 1$ we have
$$2 \<A(u)-A(v), u-v\>_V
   \le - \|u-v\|_V^2+  K\|u-v\|_H^2,\ u,v\in V.$$
\end{lem}
\begin{proof} $(1)$
 Since all $f_i$ are  Lipschitz and  linear growth, we have
\begin{align*}
& ~~~~ \<A(u)-A(v), u-v\>_V \\
 &= - \|u-v\|_V^2+ \sum_{i=1}^d \int_\Lambda \left(f_i(u)D_i u-f_i(v)D_i v\right)\left(u-v\right) d
 x\\
&= - \|u-v\|_V^2+ \sum_{i=1}^d \int_\Lambda \left(f_i(u)(D_i u-D_i
v) +D_iv( f_i(u) -f_i(v))\right)\left(u-v\right) dx\\
&\le - \|u-v\|_V^2+ \sum_{i=1}^d \bigg[ \left( \int_\Lambda (D_i
u-D_i v)^2 d x\right)^{1/2} \left( \int_\Lambda f_i^2(u)(u-v)^2 dx
\right)^{1/2}\\
&~~ +\left( \int_\Lambda (D_i v)^2 d x\right)^{1/2} \left(
\int_\Lambda \left(f_i(u)-f_i(v)\right)^2
 (u-v)^2 dx \right)^{1/2} \bigg] \\
 &\le - \|u-v\|_V^2+ K \|u-v\|_V  \left(\int_\Lambda
 \left(1+ u^4\right) dx \right)^{1/4} \left(\int_\Lambda
 (u-v)^4 dx \right)^{1/4}    + K \| v\|_V \left(\int_\Lambda
 (u-v)^4 dx \right)^{1/2} \\
&\le - \|u-v\|_V^2+ K \|u-v\|_V^{3/2}  \|u-v\|_H^{1/2}\left(1+\|u\|_{L^4} \right) + 2K \| v\|_V \|u-v\|_V
\|u-v\|_H \\
&\le - \frac{1}{2}\|u-v\|_V^2+  \left(K +K\|v\|_V^2+K\|u\|_{L^4}^4
\right)\|u-v\|_H^2, \ u,v\in V,
\end{align*}
where $K$ is a constant that may change from line to line, and  we also used
the following well known estimate on $\mathbb{R}^2$ (see \cite[Lemma
2.1]{MS02})
\begin{equation}\label{e3}
  \|u\|_{L^4}^4 \le 2   \|u\|_{L^2}^2  \|\nabla u\|_{L^2}^2, \ u\in W_0^{1,2}(\Lambda) .
\end{equation}

$(2)$ For $d=3$ we use the following estimate (cf. \cite{MS02})
\begin{equation}\label{e4}
  \|u\|_{L^4}^4 \le 4  \|u\|_{L^2}  \|\nabla u\|_{L^2}^3, \  u\in W_0^{1,2}(\Lambda),
\end{equation}
then the second assertion can be derived similarly by using Young's
inequality.

$(3)$ This assertion  obviously follows from the estimates in $(i)$.
\end{proof}

\begin{rem}\label{R2} (1)
If all $f_i$ are bounded, then the local monotonicity $(H2)$ also implies the coercivity $(H3)$.

(2) If we write the operator in the following  vector form
$$   A(u)=\Delta u+  \nabla \cdot \vec{F}(u),   $$
where $\vec{F}(x)=(F_1(x),\cdots,F_d(x)): \mathbb{R}\rightarrow \mathbb{R}^d$ satisfies
\begin{equation*}
 |\vec{F}(x)-\vec{F}(y)|\le C(1+|x|+|y|)|x-y|, \ x,y\in \mathbb{R}.
\end{equation*}
Then by using the divergence theorem (or Stokes' theorem) one can
show that for $d<3$,
$$2 \<A(u)-A(v), u-v\>_V
   \le - \|u-v\|_V^2+  \left(K +K\|u\|_{L^4}^4 + K\|v\|_{L^4}^4 \right)\|u-v\|_H^2,\ u,v\in V,$$
 for $d=3$ we have
$$2 \<A(u)-A(v), u-v\>_V
   \le - \|u-v\|_V^2+  \left(K +K\|u\|_{L^4}^8 + K\|v\|_{L^4}^8 \right)\|u-v\|_H^2,\ u,v\in V.$$
And it's also easy to show the coercivity $(H3)$ holds since we have
$$      \<\nabla \cdot \vec{F}(u), u\>_V=-\int_\Lambda \vec{F}(u)\cdot \nabla u dx=0,
     \ u\in  W_0^{1,2}(\Lambda).  $$
\end{rem}

% for $d=4$, we only have the sobolev embedding $ L^4\subseteq W_0^{1,2} $.

The first example is a general semilinear equation on $\Lambda \subseteq \mathbb{R}$,
which unifies the classical reaction-diffusion equation and Burgers equation.

\begin{exa}\label{E0}
 Consider  the following equation
\begin{equation}\label{Burges1}
 u'= \frac{\partial^2 u}{\partial x^2} + \frac{\partial F}{\partial x}(u)+g(u)+h, \ u(0)=u_0\in L^2(\Lambda).
\end{equation}
Suppose the following conditions hold for some constant $C>0$:

(i)  $F$ is a function on $\mathbb{R}$ satisfies
$$ \    |F(x)-F(y)| \le C(1+|x|+|y|)|x-y|, \ x,y\in \mathbb{R}.  $$

(ii) $g$ is a continuous function  on $\mathbb{R}$ such that
\begin{equation}\label{c0}
\begin{split}
 g(x)x  &\le C(x^2+1), \  x\in \mathbb{R};\\
|g(x)| & \le C(|x|^3+1), \  x\in \mathbb{R};\\
 (g(x)-g(y))(x-y)&\le C(1+|x|^t+|y|^t)(x-y)^2, \  x,y\in \mathbb{R},
\end{split}
     \end{equation}
where $t\ge 1$ is a constant.

(iii) $h\in W^{-1,2}(\Lambda)$.

Then  $(\ref{Burges1})$ has a  solution
$u\in W^1_2(0,T; W_0^{1,2}(\Lambda), L^2(\Lambda))$.
Moreover, if  $t\le 2$, then the solution of $(\ref{Burges1})$  is also unique.
\end{exa}

\begin{proof}
We define the Gelfand triple
$$ V:=W_0^{1,2}(\Lambda)\subseteq H:=L^2(\Lambda) \subseteq  W^{-1,2}(\Lambda)  $$
and the operator
$$A(u)=\frac{\partial^2 u}{\partial x^2} + \frac{\partial F}{\partial x}(u)+g(u), \ u\in V.  $$
It is easy to show that $(H1)$ holds by  the continuity of $F$ and $g$.

Similar to Lemma \ref{L3.1}, one can easily show that
\begin{equation}
 \begin{split}
  & \<\frac{\partial F}{\partial x}(u)- \frac{\partial F}{\partial x}(v), u-v\>_V \\
=& -\int_\Lambda  \( F(u)-F(v)   \)
\left(\frac{\partial u}{\partial x}-\frac{\partial v}{\partial x} \right) dx   \\
\le & \frac{1}{4}\|u-v\|_V^2+C\left( 1+\|u\|_{L^4}^4+\|v\|_{L^4}^4
\right
) \|u-v\|_H^2, \ u,v\in V.
 \end{split}
\end{equation}
By integration by parts formula we have
$$   \<\frac{\partial F}{\partial x}(u), u\>_V=0, \ u\in V.   $$

 (\ref{c0}) and (\ref{e3}) implies that
\begin{equation}
 \begin{split}
 & \<g(u)-g(v),u-v\>_V \\
\le & C\left(1+\|u\|_{L^{2t}}^t+ \|v\|_{L^{2t}}^t \right)\|u-v\|_{L^4}^2\\
\le & \frac{1}{4}\|u-v\|_V^2+C\left(
1+\|u\|_{L^{2t}}^{2t}+\|v\|_{L^{2t}}^{2t} \right) \|u-v\|_H^2, \
u,v\in V.
 \end{split}
\end{equation}
Therefore, we have
$$2 \<A(u)-A(v), u-v\>_V
   \le - \|u-v\|_V^2+ C \left(1  + \|u\|_{L^4}^4 +\|u\|_{L^{2t}}^{2t} +
\|v\|_{L^4}^4 +\|v\|_{L^{2t}}^{2t} \right)\|u-v\|_H^2,\ u,v\in V,$$
$i.e.$ $(H2)$ holds.

Note that by (\ref{c0}) we have
$$ \<g(u),u\>_V \le C\left(1+\|u\|_{H}^2 \right), \ u\in V,$$
hence $(H3)$ holds with $\alpha=2$.

By the Sobolev embedding theorem we have
$$  \|g(u)\|_{V^*}\le C\left(1+\|u\|_{L^3}^3\right)\le C\left(1+\|u\|_V\|u\|_{H}^{2}\right), \ u\in V, $$
$$ \|\frac{\partial F}{\partial x}(u)\|_{V^*}\le  \|F(u)\|_{H} \le
C\left(1+\|u\|_{L^4}^2\right)\le C\left(1+\|u\|_V\|u\|_{H}\right), \
u\in V.  $$ Hence $(H4)$ also holds (with $\beta=2$).

Therefore, the assertions follow from Theorem \ref{T1}.
\end{proof}

\begin{rem}
 If we take $F(x)=x^2$ and $g=h=0$, then (\ref{Burges1}) is the classical Burgers equation
in fluid mechanics.
 If we take $g(x)=x-x^3$ and $F=h=0$, then (\ref{Burges1}) is the well known reaction-diffusion equation.
\end{rem}

\begin{exa}\label{E1}
 Consider  the following equation
\begin{equation}\label{Burges}
 u'=\Delta u+ \sum_{i=1}^d f_i(u)D_i u+g(u)+h, \ u(0)=u_0\in L^2(\Lambda).
\end{equation}
Suppose the following conditions hold for some constant $C>0$:

(i)  $f_i$ are bounded  Lipschitz functions on $\mathbb{R}$ for $i=1,\cdots, d$;

(ii) $g$ is a continuous function  on $\mathbb{R}$ such that
\begin{equation}\label{c1}
\begin{split}
 g(x)x  &\le C(x^2+1), \  x\in \mathbb{R};\\
|g(x)| & \le C(|x|^r+1), \  x\in \mathbb{R};\\
 (g(x)-g(y))(x-y)&\le C(1+|x|^t+|y|^t)(x-y)^2, \  x,y\in \mathbb{R},
\end{split}
     \end{equation}
where  $r,t\ge 1$ are some constants.

(iii) $h\in W^{-1,2}(\Lambda)$.

Then we have

(1) if $d=2$, $r=\frac{7}{3}$ and $t=2$,
  $(\ref{Burges})$ has a  unique solution
$u\in W^1_2(0,T; W_0^{1,2}(\Lambda), L^2(\Lambda))$.

(2) if $d=3$, $r=\frac{7}{3}$ and $t\le 3$,   $(\ref{Burges})$ has a  solution
$u\in W^1_2(0,T; W_0^{1,2}(\Lambda), L^2(\Lambda))$.
Moreover, if  $t=\frac{4}{3}$, $f_i, i=1,2,3$ are bounded measurable functions on $\Lambda$ and independent of $u$,
  then the solution of $(\ref{Burges})$
is also unique.
\end{exa}

\begin{proof} We define the Gelfand triple
$$ V:=W_0^{1,2}(\Lambda)\subseteq H:=L^2(\Lambda) \subseteq  W^{-1,2}(\Lambda)  $$
and the operator
$$ A(u)=\Delta u+ \sum_{i=1}^d f_i(u)D_i u+g(u), \ u\in V.  $$
By assumption (\ref{c1}) we have
$$  \<g(u)-g(v),u-v\>_V
\le  C\left(1+\|u\|_{L^{2t}}^t+ \|v\|_{L^{2t}}^t \right)\|u-v\|_{L^4}^2\\.$$
Then from (\ref{e3}) or (\ref{e4}) and  Lemma \ref{L3.1} we have for $d=2$
$$2 \<A(u)-A(v), u-v\>_V
   \le - \|u-v\|_V^2+  K\left(1  +\|v\|_V^2 + \|u\|_{L^{2t}}^{2t}+ \|v\|_{L^{2t}}^{2t}  \right)\|u-v\|_H^2,\ u,v\in V,$$
and for $d=3$,
$$2 \<A(u)-A(v), u-v\>_V
   \le - \|u-v\|_V^2+
K\left(1  +\|v\|_V^4 + \|u\|_{L^{2t}}^{4t}+ \|v\|_{L^{2t}}^{4t}  \right)\|u-v\|_H^2,\ u,v\in V,$$
i.e. $(H2)$ holds.

Note that
$$ \<g(u),u\>_V \le C\left(1+\|u\|_{H}^2 \right), \ u\in V.$$
Then by Lemma \ref{L3.1} and Remark \ref{R2}  we know that $(H3)$ holds with  $\alpha=2$.

% For $d=1,2$ we have
% $$\|\Delta u+ \sum_{i=1}^d f(u) D_i u \|_{V^*} \le \|u\|_V +C\left(1+\|u\|_{L^4}^2\right)
% \le C\left(1+\|u\|_V\right)\left(1+\|u\|_H\right), \ u\in V. $$

For $d=2,3$ we have
 $$  \|g(u)\|_{V^*}\le C\left(1+\|u\|_{L^{6r/5}}^{r}\right), \ u\in V. $$
% hence $(H4)$ follows from $(\ref{e3})$ and .

For  $r= \frac{7}{3}$, by the interpolation theorem we have
$$ \|u\|_{L^{6r/5}}\le \|u\|_{L^2}^{4/7} \|u\|_{L^6}^{3/7}, \ u\in W_0^{1,2}(\Lambda) \subseteq L^6(\Lambda).  $$
Then
$$  \|g(u)\|_{V^*}\le C\left(1+\|u\|_{L^{6r/5}}^{r}\right)\le C\left(1+
\|u\|_H^{4/3} \|u\|_{V} \right), \ u\in V.  $$
Hence $(H4)$ holds.

The hemicontinuity $(H1)$ follows easily from the continuity of $f$ and $g$.

Therefore, all assertions follow from Theorem \ref{T1}.

In particular, if $d=3$ and $f_i, i=1,2,3$ are bounded measurable functions on $\Lambda$ and independent of $u$,
 then we have
$$2 \<A(u)-A(v), u-v\>_V
   \le - \|u-v\|_V^2+
K\left(1  + \|u\|_{L^{2t}}^{4t}+ \|v\|_{L^{2t}}^{4t}
\right)\|u-v\|_H^2,\ u,v\in V.$$ Since $t=\frac{4}{3}$, by the
interpolation inequality we have
$$   \|u\|_{L^{2t}} \le   \|u\|_{L^2}^{5/8}   \|u\|_{L^{6}}^{3/8}, \ u\in V.  $$
Therefore
$$   \|u\|_{L^{2t}}^{4t} \le  C \|u\|_{H}^{10/3}   \|u\|_{V}^{2}, \ u\in V.  $$
Hence the solution of of $(\ref{Burges})$
is  unique.
\end{proof}

\begin{rem}
(1) As we mentioned in Remark 1.1, the classical result for monotone
operators can not be applied to the above example. The typical
example of monotone perturbation is to assume all $f_i$ are
independent of unknown solution $u$, $g$ is monotone ($e.g.$
Lipschitz) and has linear growth ($r=1$). However, here we allow $g$
is locally monotone ($e.g.$ locally Lipschitz) and has certain
polynomial growth ($r>1$).

(2) The boundedness of $f_i$ is only assumed in order to verify the coercivity $(H3)$.
This assumption can be removed if we formulate (\ref{Burges}) in vector form
as explained in Remark \ref{R2}.
\end{rem}

We may also consider the following quasi-linear evolution equations on
$\mathbb{R}^d\ (d\ge 3)$.

\begin{exa} Consider  the Gelfand triple
$$ V:=W_0^{1,p}(\Lambda)\subseteq H:=L^2(\Lambda) \subseteq  W^{-1,q}(\Lambda) $$
and the following equation on $\mathbb{R}^d$ for $p> 2$
\begin{equation}\label{p-Laplace}
 u'=\sum_{i=1}^d D_i\left(|D_iu|^{p-2} D_i u  \right) +g(u)+h, \ u(0)=u_0\in L^2(\Lambda).
\end{equation}
Suppose the following conditions hold:

(i) $g$ is a continuous function  on $\mathbb{R}$ such that
\begin{equation}
\begin{split}\label{c2}
 g(x)x  &\le C(|x|^{\frac{p}{2}+1}+1), \  x\in \mathbb{R};\\
|g(x)| & \le C(|x|^{r}+1), \  x\in \mathbb{R};\\
(g(x)-g(y))(x-y)&\le C(1+|x|^t+|y|^t)|x-y|^{s}, \  x,y\in \mathbb{R},
\end{split}
\end{equation}
where  $C>0$ and $r,s,t\ge 1$ are some constants.

(ii) $h\in W^{-1,q}(\Lambda)$, $p^{-1}+q^{-1}=1$.

Then we have

(1) if $d<p$, $s=2$ and $r= p+1$,  $(\ref{p-Laplace})$ has a solution.
Moreover, if $t\le p$ also holds, then the solution is unique.

(2) if $d>p$, $2<s<p$, $r=\frac{2p}{d}+p-1$ and $t\le \frac{p^2(s-2)}{(d-p)(p-2)}$,
 $(\ref{p-Laplace})$ has a solution.
The solution is  unique if $t\le \frac{p(p-s)}{p-2}$ also holds.
\end{exa}

\begin{proof}
(1) It's well known that $\sum_{i=1}^d D_i\left(|D_iu|^{p-2} D_i u\right) $
satisfy $(H1)$-$(H4)$ (cf. \cite{L08,L08b}). In particular,
 there exists a constant $\delta>0$ such that
\begin{equation}\label{e7}
 \sum_{i=1}^d\<  D_i\left(|D_iu|^{p-2} D_i u\right)-  D_i\left(|D_iv|^{p-2} D_i v\right)  , u-v  \>_V
\le - \delta \|u-v\|_V^p, \ u,v\in  W_0^{1,p}(\Lambda).
\end{equation}

Recall that  for $d<p$ we have the following Sobolev embedding
$$      W_0^{1,p}(\Lambda) \subseteq   L^{\infty}(\Lambda).   $$
Hence we have
\begin{equation}
\begin{split}
  \<g(u)-g(v),u-v\>_V & \le C \int_\Lambda \left(1+|u|^t+ |v|^t \right) |u-v|^2 d x\\
  &\le C \left( 1+ \|u\|_{L^{\infty}}^t +\|v\|_{L^{\infty}}^t   \right) \|u-v\|^2_{L^{2}}\\
  &\le C \left( 1+ \|u\|_{V}^t +\|v\|_{V}^t   \right) \|u-v\|_{H}^{2}, \ u,v\in V,
\end{split}
\end{equation}
where $C$ is a constant may change from line to line.

Hence $(H2)$ holds.

Note that from (\ref{c2}) we have
\begin{equation}
\begin{split}
  \<g(u), u\>_V & \le C \int_\Lambda (1+ |u|^{\frac{p}{2}+1})dx\\
       & \le C\left(1+\|u\|_{L^\infty}^{p/2}\|u\|_H\right)\\
   & \le \frac{\delta}{2} \|u\|_V^p + C\left(1+\|u\|_H^2\right), \ u\in V.
\end{split}
\end{equation}
Therefore, $(H3)$ holds with $\alpha=p$ by (\ref{e7}).

$(H4)$ follows from the following estimate:
$$  \|g(u)\|_{V^*} \le C\left(1+\|u\|_{L^{p+1}}^{p+1}\right)\le C\left(1+\|u\|_{L^{\infty}}^{p-1}\|u\|_{H}^2\right)
, \ u\in V. $$

Hence the assertions follow from Theorem \ref{T1}.

(2)
Note that  for $d>p$ we have the following Sobolev embedding
$$      W_0^{1,p}(\Lambda) \subseteq   L^{p_0}(\Lambda), \ p_0=\frac{dp}{d-p}.   $$
Let $t_0=\frac{p(s-2)}{s(p-2)}\in(0,1)$ and $p_1\in(2,p_0)$ such that
$$     \frac{1}{p_1}=\frac{1-t_0}{2}+ \frac{t_0}{p_0}.    $$
Then we have the following interpolation inequality
$$   \|u\|_{L^{p_1}}\le   \|u\|_{L^{2}}^{1-t_0}   \|u\|_{L^{p_0}}^{t_0}, \ u\in W_0^{1,p}(\Lambda).   $$
Since $2<s<p$, it is easy to show that $s<p_1$.

Let $p_2=\frac{p_1}{p_1-s}$, then by assumption (\ref{c2}) we have
\begin{equation}\label{e8}
\begin{split}
  \<g(u)-g(v),u-v\>_V & \le C \int_\Lambda \left(1+|u|^t+ |v|^t \right) |u-v|^s d x\\
  &\le C \left( 1+ \|u\|_{L^{tp_2}}^t +\|v\|_{L^{tp_2}}^t   \right) \|u-v\|^s_{L^{p_1}}\\
  &\le C \left( 1+ \|u\|_{L^{tp_2}}^t +\|v\|_{L^{tp_2}}^t   \right) \|u-v\|_{L^{2}}^{s(1-t_0)}   \|u-v\|_{L^{p_0}}^{st_0}\\
&\le \varepsilon \|u-v\|_{L^{p_0}}^{p} +
C_\varepsilon \left( 1+ \|u\|_{L^{tp_2}}^{tb} +\|v\|_{L^{tp_2}}^{tb}   \right)   \|u-v\|_{L^{2}}^{2}
,
\end{split}
\end{equation}
where $\varepsilon, C_\varepsilon$ are some constants and the last step follows from the following Young inequality
$$  xy\le \varepsilon x^a +C_\varepsilon y^b, \ x,y\in\mathbb{R},\  a=\frac{p-2}{s-2},\ b=\frac{p-2}{p-s}. $$
By calculation we have
$$  \frac{s}{p_1}=\frac{p-s}{p-2}+\frac{p(s-2)}{p_0(p-2)},\ p_2=\frac{p_0(p-2)}{(p_0-p)(s-2)}. $$
Hence if $t\le \frac{(p_0-p)(s-2)}{p-2}$, then
$$ \|u\|_{L^{tp_2}}\le C \|u\|_{L^{p_0}} \le C  \|u\|_{V}, \ v\in V.   $$
Therefore, $(H2)$ follows from (\ref{e7}) and (\ref{e8}).

$(H3)$ can be verified for $\alpha=p$ in a similar way.

For $r=\frac{2p}{d}+p-1$, by the interpolation inequality we have
$$\|g(u)\|_{V^*}\le C\left(1+ \|u\|_{L^{rp_0^\prime}}^r \right)
\le C\left( 1+\|u\|_{p_0}^{p-1}\|u\|_H^{\beta} \right), \ u\in V, $$
where
$$ \frac{1}{p_0}+\frac{1}{p_0^\prime},\ \ \beta=\frac{2p}{d}.  $$
Therefore, $(H4)$ also holds.

Then all assertions  follow from Theorem \ref{T1}.
\end{proof}

\begin{rem}
One further generalization is to replace  $\sum_{i=1}^d
D_i\left(|D_iu|^{p-2} D_i u\right)$ by more general quasi-linear
differential operator
$$ \sum_{|\alpha|\le m} (-1)^{|\alpha|}D_\alpha A_\alpha(x,Du(x,t);t),  $$
where $Du=(D_\beta u)_{|\beta|\le m}$. Under certain assumptions (cf. \cite[Proposition 30.10]{Z90}) this operator
satisfies the monotonicity and coercivity condition.

According to Theorem \ref{T1}, we can  obtain the existence and
uniqueness of solutions to this type of quasi-linear PDE with some
non-monotone perturbations ($e.g.$ some locally Lipschitz lower
order terms).
\end{rem}

Now we apply Theorem \ref{T1} to the Navier-Stokes equation.

Let $\Lambda$ be a bounded domain in $\mathbb{R}^2$ with smooth
boundary. It's well known that by means of divergence free Hilbert
spaces $V,H$ and the Helmhotz-Leray orthogonal projection $P_H$, the
classical form of the Navier-Stokes equation can be formulated in
the following form:
\begin{equation}\label{NSE}
u'=Au+B(u)+f,\ u(0)=u_0\in H,
\end{equation}
where
$$ V=\left\{ v\in W_0^{1,2}(\Lambda,\mathbb{R}^2): \nabla \cdot v=0 \ a.e.\ \text{in} \ \Lambda   \right\}, \
\|v\|_V:=\left(\int_\Lambda |\nabla v|^2 dx  \right)^{1/2},
$$
and $H$ is the closure of $V$ in the following norm
$$ \|v\|_H:=\left(\int_\Lambda | v|^2 dx  \right)^{1/2}.$$
The linear operator $P_H$ (Helmhotz-Leray projection) and $A$
(Stokes operator with viscosity constant $\nu$) are defined by
$$ P_H: L^2(\Lambda, \mathbb{R}^2)\rightarrow H,\  \text{ orthogonal projection}; $$
$$  A:=W^{2,2}(\Lambda, \mathbb{R}^2)\cap V\rightarrow H, \ Au=\nu P_H \Delta u,      $$
and the nonlinear operator
$$ B:  \mathcal{D}_B\subset H\times V\rightarrow H, \ B(u,v)=- P_H\left[(u \cdot \nabla) v\right], B(u)=B(u,u).  $$

It's well known that by using the Gelfand triple
$$     V\subseteq H\equiv H^*\subseteq V^*   $$
the following mappings
$$ A: V\rightarrow V^*, \  B: V\times V\rightarrow V^*  $$
are well defined. In particular, we have
$$ \<B(u,v),w\>_V=-\<B(u,w),v\>_V, \   \<B(u,v),v\>_V=0,\  u,v,w\in V. $$

\begin{exa}(2D Navier-Stokes equation)
For $f\in L^2(0,T;V^*)$ and $u_0\in H$,
 $(\ref{NSE})$ has a unique solution.
\end{exa}
\begin{proof} The hemicontinuity $(H1)$ is easy to show since $B$ is a bilinear map.

Note that $ \<B(v),v\>_V=0$, it's also easy to get the coercivity $(H3)$
$$ \<Av+B(v)+f,v\>_V\le -\nu\|v\|_V^2+\|f\|_{V^*}\|v\|_V \le
  -\frac{\nu}{2}\|v\|_V^2+C\|f\|_{V^*}^2, \ v\in V.   $$
Recall the following estimate (cf. \cite[Lemma 2.1, 2.2]{MS02})
\begin{equation}\label{e2}
 \begin{split}
  |\<B(w),v\>_V| &\le 2 \|w\|_{L^4(\Lambda;\mathbb{R}^2)}\|v\|_V; \\
  |\<B(w),v\>_V| &\le 2  \|w\|_V^{3/2} \|w\|_H^{1/2}  \|v\|_{L^4(\Lambda;\mathbb{R}^2)}, v,w\in V. \\
 \end{split}
\end{equation}
Then we have
\begin{equation}
 \begin{split}
  \<B(u)-B(v),u-v\>_V &=-  \<B(u,u-v),v\>_V+  \<B(v,u-v),v\>_V \\
 &= -  \<B(u-v),v\>_V \\
 &\le 2  \|u-v\|_V^{3/2} \|u-v\|_H^{1/2}  \|v\|_{L^4(\Lambda;\mathbb{R}^2)} \\
& \le \frac{\nu}{2} \|u-v\|_V^{2} + \frac{32}{\nu^3}  \|v\|_{L^4(\Lambda;\mathbb{R}^2)}^4 \|u-v\|_H^{2},
\ u,v\in V.
 \end{split}
\end{equation}
Hence we have the local monotonicity $(H2)$
$$  \<Au+B(u)-Av-B(v),u-v\>_V
\le -\frac{\nu}{2} \|u-v\|_V^{2} + \frac{32}{\nu^3}  \|v\|_{L^4(\Lambda;\mathbb{R}^2)}^4 \|u-v\|_H^{2}.  $$

The growth $(H4)$ follows from (\ref{e2}) and (\ref{e3}).

Hence the existence of solution to (\ref{NSE}) follows from Theorem \ref{T1}.

By definition any solution $u$ of  (\ref{NSE}) is a element in $L^2([0,T];V)$ and $C([0,T];H)$, then (\ref{e3})
implies that
$$  \int_0^T \|u(t)\|_{L^4}^4 d t \le 2\sup_{t\in[0,T]}\|u(t)\|_H^2 \int_0^T \|u(t)\|_{V}^2 d t < \infty.$$
Hence the solution of (\ref{NSE}) is also  unique.
\end{proof}

\begin{rem} (1)
The main result can be also applied to some other classes of two
dimensional hydrodynamical models such as magneto-hydrodynamic
equations, the Boussinesq model for the B\'{e}nard convection and 2D
magnetic B\'{e}nard problem. We refer to \cite{CM10} (and the
references therein) for the details of these models. Note that the
assumption $(C1)$ in \cite{CM10} implies a special type of local
monotonicity (e.g. see (2.8) in \cite{CM10}).

(2)
 For the 3D Navier-Stokes equation, we recall the following
estimate (cf. \cite[(2.5)]{MS02})
$$   \|\psi\|_{L^4}^4\le 4 \|\psi\|_{L^2} \|\nabla \psi\|_{L^2}^3, \ \psi\in W_0^{1,2}(\Lambda; \mathbb{R}^3).   $$
Then  one can show  that
\begin{equation}
 \begin{split}
  \<B(u)-B(v),u-v\>_V &= -  \<B(u-v),v\>_V \\
 &\le 2  \|u-v\|_V^{7/4} \|u-v\|_H^{1/4}  \|v\|_{L^4(\Lambda;\mathbb{R}^3)} \\
& \le \frac{\nu}{2} \|u-v\|_V^{2} + \frac{2^{12}}{\nu^7}  \|v\|_{L^4(\Lambda;\mathbb{R}^3)}^8 \|u-v\|_H^{2},
\ u,v\in V.
 \end{split}
\end{equation}
Hence we have the following local monotonicity $(H2)$
$$  \<Au+B(u)-Av-B(v),u-v\>_V
 \le -\frac{\nu}{2} \|u-v\|_V^{2} + \frac{2^{12}}{\nu^7}  \|v\|_{L^4(\Lambda;\mathbb{R}^3)}^8 \|u-v\|_H^{2}.  $$
However, we only have  the following growth condition  in the 3D
case:
$$ \|B(u)\|_{V^*}\le 2 \|u\|_{L^4(\Lambda; \mathbb{R}^3)}^2\le 4\|u\|_H^{1/2}\|u\|_V^{3/2}, \ u\in V.$$
Unfortunately, this is not enough to verify  $(H4)$ in Theorem
\ref{T1}.
% $$    \|B(u)\|_{V^*}\le C(1+ \|u\|_V) (1+ \|u\|_H^{\beta}),  \ u\in V.    $$

% By interpolation inequality we have
% $$ \|u\|_{L^4(\Lambda; \mathbb{R}^3)}^2\le \|u\|_{L^2(\Lambda; \mathbb{R}^3)}^{1/2}
% \|u\|_{L^6(\Lambda; \mathbb{R}^3)}^{3/2} . $$
\end{rem}

Now we apply the main result to the 3D Leray-$\alpha$ model of
turbulence, which is a regularization of the 3D Navier-Stokes
equation. It was first considered by Leray \cite{L34} in order to
prove the existence of a solution to the Navier-Stokes equation in
$\mathbb{R}^3$. Here we use a special smoothing kernel in the 3D
Leray-$\alpha$ model, which was first considered in \cite{CH05} (cf.
\cite{CTV07} for more references). It has been shown in \cite{CH05}
that the 3D Leray-$\alpha$ model compares successfully with
empirical data from turbulent channel and pipe flows for a wide
range of Reynolds numbers. This model has a great potential to
become a good sub-grid-scale large-eddy simulation model of
turbulence. The Leray-$\alpha$ model can be formulated as follows:
\begin{equation}\begin{split}\label{3D}
& u^\prime=\nu \Delta u-(v\cdot \nabla)u-\nabla p+f, \\
&  \nabla\cdot u=0,\ u=v-\alpha^2 \Delta v
\end{split}
\end{equation}
 where $\nu>0$ is the viscosity, $u$ is the velocity, $p$ is the pressure and $f$ is a given
 body-forcing term.

By using the same divergence free Hilbert spaces $V, H$ (but in 3D)
one can rewrite the Leray-$\alpha$ model into the following abstract
form:
\begin{equation}\label{Leray}
u^\prime=Au+B(u,u)+f, \ u(0)=u_0\in H,
\end{equation}
where
$$   Au=\nu P_H\Delta u, B(u,v)=-P_H\left[\left( \left(I-\alpha^2\Delta\right)^{-1}u \cdot \nabla   \right)v  \right].  $$

\begin{exa}(3D Leray-$\alpha$ model)
For $f\in L^2(0,T;V^*)$ and $u_0\in H$,
 $(\ref{Leray})$ has a unique solution.
\end{exa}
\begin{proof} $(H1)$ holds obviously since $B$ is a bilinear map.

Note that $ \<B(u,v),v\>_V=0$, it's also easy to get the coercivity
$(H3)$:
$$ \<Av+B(v,v)+f,v\>_V\le -\nu\|v\|_V^2+\|f\|_{V^*}\|v\|_V \le
  -\frac{\nu}{2}\|v\|_V^2+C\|f\|_{V^*}^2, \ v\in V.   $$
Recall the following well-known estimate (cf. \cite[Lemma 2.1,
2.2]{MS02})
\begin{equation}\label{e5}
 \begin{split}
  & |\<B(u,v),w\>_V|\\
   \le& c \|(I-\alpha^2\Delta)^{-1}u\|_{H}^{1/4} \|(I-\alpha^2\Delta)^{-1}u\|_{V}^{3/4}
   \|v\|_{H}^{1/4} \|v \|_{V}^{3/4} \|w\|_V \\
   \le &  C   \|u\|_H \|v\|_{H}^{1/4} \|v \|_{V}^{3/4} \|w\|_V, \  u, v,w\in
   V,
 \end{split}
\end{equation}
where $c, C$ are some constants.

Then we have
\begin{equation}
 \begin{split}
 & \<B(u,u)-B(v,v),u-v\>_V \\
 =&-  \<B(u,u-v),v\>_V+  \<B(v,u-v),v\>_V \\
 =& -  \<B(u-v, u-v),v\>_V \\
 \le& C  \|u-v\|_H^{5/4} \|u-v\|_V^{3/4}  \|v\|_{V} \\
 \le& \frac{\nu}{2} \|u-v\|_V^{2} + C_\nu \|v\|_{V}^{8/5}
\|u-v\|_H^{2}, \ u,v\in V.
 \end{split}
\end{equation}
Hence we have the local monotonicity $(H2)$:
$$  \<Au+B(u,u)-Av-B(v,v),u-v\>_V
\le -\frac{\nu}{2} \|u-v\|_V^{2} + C_\nu \|v\|_{V}^{8/5}
\|u-v\|_H^{2}.
$$

Note that (\ref{e5}) also implies that $(H4)$ holds.

Hence the existence and uniqueness of solutions to (\ref{Leray})
follows from Theorem \ref{T1}.
\end{proof}

\begin{rem}
The main result can also be applied to some other equations such as
  3D Tamed Navier-Stokes equation, which is also a modified version of 3D Navier-Stokes
  equation. One may refer to \cite{LR10,RZ09} for
  more details.
\end{rem}

\section*{Acknowledgements}
 The author would like to thank Professors Phillipe Cl\'{e}ment, Michael R\"{o}ckner and  Feng-yu Wang
 for their valuable discussions.

%------------------------------------------------------------------------------------------------------------------------
%--------------------------------------      References     ------------------------------------------------
%----------------------------------------------------------------------------------------------------------------------
%\input{bib}

\end{document}